\documentclass[12pt,letterpaper]{amsart}

\usepackage{amsthm}
\usepackage{tabularx}
\usepackage{tikz}
\usepackage{parskip}

\usepackage[english]{babel} 
\usepackage[latin1]{inputenc}
\usepackage{amsmath, bm} 
\usepackage{amsfonts}
\usepackage{amssymb}
\usepackage{stmaryrd}
\usepackage{latexsym} 
\usepackage{graphicx}
\usepackage{subfigure}
\usepackage{hyperref}
\usepackage{verbatim}
\usepackage[all]{xy}
\usepackage{graphics}
\usepackage{pdfsync}
\usepackage{xcolor}

\usepackage{listings}
\usepackage[noend]{algpseudocode}
\usepackage{csquotes}
\usepackage{ytableau}
\usepackage{textcmds}
\usepackage{latexsym}

\newtheorem{theorem}{Theorem}[section]
\newtheorem{lemma}[theorem]{Lemma}
\newtheorem{corollary}[theorem]{Corollary}

\theoremstyle{definition}
\newtheorem{definition}[theorem]{Definition}
\newtheorem{example}[theorem]{Example}

\newlength{\cellsize}
\newcommand\tableau[1]{
\vcenter{
\let\\=\cr
\baselineskip=-16000pt
\lineskiplimit=16000pt
\lineskip=0pt
\halign{&\tableaucell{##}\cr#1\crcr}}}

\newcommand{\tableaucell}[1]{{%
\def \arg{#1}\def \void{}%
\ifx \void \arg
\vbox to \cellsize{\vfil \hrule width \cellsize height 0pt}%
\else
\unitlength=\cellsize
\begin{picture}(1,1)
\put(0,0){\makebox(1,1)[c]{$#1$}}
\put(0,0){\line(1,0){1}}
\put(0,1){\line(1,0){1}}
\put(0,0){\line(0,1){1}}
\put(1,0){\line(0,1){1}}
\end{picture}%
\fi}}

\DeclareMathOperator{\sgn}{sgn}

\DeclareMathOperator{\comp}{comp}

\DeclareMathOperator{\set}{set}
\newcommand{\dI}{\mathfrak{S}^*}
\newcommand{\rdI}{\mathcal{R}\mathfrak{S}^*}

\DeclareMathOperator{\QSym}{QSym}
\DeclareMathOperator{\NSym}{NSym}
\DeclareMathOperator{\Des}{Des}

\DeclareMathOperator{\negc}{neg}
\DeclareMathOperator{\tail}{tail}

\newcommand{\Qsym}{\ensuremath{\operatorname{QSym}}}

\newcommand{\Nsym}{\ensuremath{\operatorname{NSym}}}
\newcommand{\nce}{\mathbf{e}}         	
\newcommand{\nch}{\mathbf{h}}         	
\newcommand{\nci}{{\mathfrak{S}}}	         
\newcommand{\ncri}{\mathcal{R}{\mathfrak{S}}} 

\savebox2{%
\begin{picture}(15,15)
\put(0,0){\line(1,0){15}}
\put(0,0){\line(0,1){15}}
\put(15,0){\line(0,1){15}}
\put(0,15){\line(1,0){15}}
\end{picture}}
\newcommand\cellify[1]{\def\thearg{#1}\def\nothing{}%
\ifx\thearg\nothing
\vrule width0pt height\cellsize depth0pt\else
\hbox to 0pt{\usebox2\hss}\fi%
\vbox to 15\unitlength{
\vss
\hbox to 15\unitlength{\hss$#1$\hss}
\vss}}
\newcommand\expath[1]{%
\hbox to 0pt{\usebox3\hss}%
\vbox to 15\unitlength{
\vss
\hbox to 15\unitlength{\hss$#1$\hss}
\vss}}
\newcommand\bas[1]{\omit \vbox to \cellsize{ \vss \hbox to \cellsize{\hss$#1$\hss} \vss}}

\newcommand{\I}{{\mathfrak{S}}}

\newcommand{\svw}[1]{\textcolor{black}{#1}}

\newcommand{\sw}[1]{\textcolor{black}{#1}}

\author[Niese, Sundaram, van Willigenburg, Wang]{Elizabeth Niese, Sheila Sundaram,\\Stephanie van Willigenburg, Shiyun Wang}
\address{Elizabeth Niese: Marshall University, Huntington, WV 25755, USA} 
\email{elizabeth.niese@gmail.com}
\address{Sheila Sundaram: Pierrepont School, Westport, CT 06880, USA}
\email{shsund@comcast.net}
\address{Stephanie van Willigenburg: University of British Columbia, Vancouver, BC V6T 1Z2, Canada}
\email{steph@math.ubc.ca}
\address{Shiyun Wang: University of Southern California, Los Angeles, CA 90089-2532, USA}
\email{shiyunwa@usc.edu}
\title[skew dual immaculate Pieri rules]{Pieri rules for skew dual immaculate functions}
\date{\today}

\begin{document}

\maketitle

\begin{abstract}
In this paper we give Pieri rules for skew dual immaculate functions and their recently discovered row-strict counterparts. We establish our rules using a right-action analogue of the skew Littlewood-Richardson rule for Hopf algebras of Lam-Lauve-Sottile. We also obtain Pieri rules for row-strict (dual) immaculate functions.
\end{abstract}


\section{Introduction}\label{sec:intro}
Schur-like functions are a new and flourishing area since the discovery of quasisymmetric Schur functions in 2011 \cite{HLMvW2011}, which led to numerous other similar  functions being discovered, for example \cite{ALvW2021, AS2019, BBSSZ2014, CFLSX2014, LMvW2013, MN2015, MN2018, MR2014}. In essence, Schur-like functions are functions that refine the ubiquitous Schur functions and reflect many of their properties, such as their combinatorics \cite{AHM2018, BLvW2011}, their representation theory \cite{BS2021, BBSSZ2015, S2020, TvW2015}, and in the case of quasisymmetric Schur functions have already been applied to resolve conjectures \cite{LM2011}. Of the various Schur-like functions to arise after the quasisymmetric Schur functions, two were naturally related to them: the dual immaculate functions \cite{BBSSZ2014} and the row-strict quasisymmetric Schur functions \cite{MR2014}. Recently a fourth basis that interpolates between these latter two bases, the \emph{row-strict dual immaculate functions}, was discovered \cite{NSvWVW2022}, thus completing the picture. The representation theory of these functions was revealed in \cite{NSvWVW2022b}, in addition to the fundamental combinatorics in \cite{NSvWVW2022}. In this paper we extend the combinatorics to uncover skew Pieri rules in the spirit of \cite{AM2011, LLS2011, TvW2018} for both row-strict and classical dual immaculate functions. 

More precisely, our paper is structured as follows. In Section~\ref{sec:rightHopfLR} we establish a right-action analogue of  \cite[Theorem 2.1]{LLS2011} in Theorem~\ref{thm:genskewLR}. We then recall required background for the Hopf algebras of quasisymmetric functions, $\Qsym$, and noncommutative symmetric functions, $\Nsym$, in Section~\ref{sec:QSymNSym}. Finally, in Section~\ref{sec:Pieri} we give (left) Pieri rules for row-strict immaculate functions and row-strict dual immaculate functions  in Corollaries~\ref{cor:rsimmleftpieri} and \ref{cor:rsdimmleftpieri}, respectively. Our final theorem is  Theorem~\ref{thm:skewpieri}, in which we establish Pieri rules for skew dual immaculate functions, and row-strict skew dual immaculate functions.

\section{The right-action skew Littlewood-Richardson rule for Hopf algebras}\label{sec:rightHopfLR}
We begin by recalling and deducing general Hopf algebra results that will be useful later.
Following Tewari and van \svw{Willigenburg}~\cite{TvW2018}, let $H$ and $H^*$ be a pair of dual Hopf algebras over a field $k$ with duality pairing $\langle \ , \ \rangle: H \otimes H^* \rightarrow k$ for which the structure of $H^*$ is dual to that of $H$ and vice versa. Let $h\in H, a\in H^*$. By Sweedler notation, we have coproduct denoted by $\Delta h=\sum h_1\otimes h_2$, and similarly $h_1h_2 = h_1\cdot h_2$ denotes product. We define \svw{the} action of one algebra on the other one \svw{by the following.} 
\begin{equation}\label{eq:action1}
h\rightharpoonup a=\sum\langle h, a_2\rangle a_1
\end{equation}  
\begin{equation}a\rightharpoonup h=\sum\langle h_2, a\rangle h_1
\end{equation}

Let $S:H\rightarrow H$ denote the antipode map.  Then for $\Delta h= \sum h_1\otimes h_2$, 
\begin{equation}\label{eq:antipode}\sum(Sh_1)h_2=\varepsilon(h)1_H=\sum h_1(Sh_2),\end{equation}where $\varepsilon$ and $1$ denote counit and unit, respectively.
Following Montgomery \cite{Montgomery1993}, we can define the convolution product $*$ for $f$ and $g$ in $H$ by
\[(f * g)(a)=\sum \langle f,a_1\rangle \langle g,a_2\rangle =\langle fg,a\rangle.\]
Then it follows that 
\[\langle g,f\rightharpoonup a\rangle=\langle gf,a\rangle.\]

Similarly, $\langle a\rightharpoonup f,b\rangle = \langle f,ba\rangle.$
Since $H^*$ is a left $H$-module algebra under $\rightharpoonup$, we have that 
\[h\rightharpoonup(a\cdot b)=\sum (h_1\rightharpoonup a)\cdot(h_2\rightharpoonup b).\]

\begin{lemma}(\cite{LLS2011})\label{lem:antipode}
For $g,h\in H$ and $a\in H^*$, \[(a\rightharpoonup g)\cdot h=\sum(S(h_2)\rightharpoonup a)\rightharpoonup (g\cdot h_1)\]
where $S: H\rightarrow H$ is the antipode.
\end{lemma}

As in Montgomery \cite{Montgomery1993}, define a right action by \svw{the following}.
\begin{equation}\label{eq:rightaction}
h\leftharpoonup a=\sum \langle h,a_1\rangle a_2
\end{equation}
\begin{equation}\label{eq:rightaction2}
a\leftharpoonup h = \sum \langle h_1,a\rangle h_2
\end{equation}
As before, it follows that $\langle g,f\leftharpoonup a\rangle=\langle fg,a\rangle$ and $\langle a\leftharpoonup f,b\rangle=\langle f,ab\rangle$.

\begin{lemma}\label{lem:rightactprod} Let $f \in H$ and $a,b \in H^*$.  Then 
\[f\leftharpoonup a\cdot b= \sum (f_1\leftharpoonup a)\cdot (f_2 \leftharpoonup b).\]
\end{lemma}

\begin{proof} 
Let $f,g \in H$ and $a,b \in H^*$.  Then 
\begin{align*}
\langle g,f\leftharpoonup (a\cdot b)\rangle &= \langle fg,ab\rangle \\
&=\langle a \leftharpoonup (fg),b\rangle\\
&=\sum\langle f_1g_1,a\rangle \langle f_2g_2,b\rangle\\
&=\sum \langle g_1,f_1\leftharpoonup a\rangle \langle g_2,f_2\leftharpoonup b\rangle \\
&=\sum \langle g, (f_1\leftharpoonup a)\cdot (f_2\leftharpoonup b)\rangle.
\end{align*}
Thus $f \leftharpoonup a\cdot b = \sum (f_1\leftharpoonup a)\cdot (f_2\leftharpoonup b)$.\qedhere
\end{proof}

\begin{lemma}\label{lem:id}
Let $a \in H^*$.  Then 
\[\varepsilon(h)\cdot 1_H \leftharpoonup a = a\] for any $h \in H$.  
\end{lemma}
\begin{proof}
Let $a \in H^*$ and $h\in H$.  Then 
\[\varepsilon(h)\cdot 1_H \leftharpoonup a  = \sum \langle \varepsilon(h)\cdot 1_H,a_1\rangle a_2.
\]
This is only nonzero when $a_1=1_{H^*}$.  
\end{proof}

\begin{lemma}\label{lem:product} Let $h \in H$ and $a,b \in H^*$.  Then 
\[a\cdot(h\leftharpoonup b) = \sum h_1 \leftharpoonup ( (S(h_2)\leftharpoonup a)\cdot b).\]
\end{lemma}
\begin{proof}
Expand the sum using Lemma~\ref{lem:rightactprod} and coassociativity, $(\Delta \otimes 1)\circ \Delta(h) = (1\otimes \Delta)\circ \Delta(h) = \sum h_1\otimes h_2\otimes h_3$, to get 
\begin{align*}
\sum h_1 \leftharpoonup ( (S(h_2)\leftharpoonup a)\cdot b)&=\sum (h_1\leftharpoonup (S(h_2)\leftharpoonup a))\cdot (h_3\leftharpoonup b)\\
&=\sum(h_1\cdot S(h_2)\leftharpoonup a) \cdot (h_3\leftharpoonup b) \text{ since $H^*$ is an $H$-module}\\
&=((\varepsilon(h)\cdot 1_H)\leftharpoonup a) \cdot  (h\leftharpoonup b) \text{ by \eqref{eq:antipode}}\\
&=a \cdot (h\leftharpoonup b) \text{ by Lemma~\ref{lem:id}}. \qedhere
\end{align*}
\end{proof}

\begin{lemma} \label{lem:raction}
Let $g,h \in H$ and $a \in H^*$.  Then 
\[ h\cdot (a\leftharpoonup g) = \sum (S(h_2) \leftharpoonup a) \leftharpoonup h_1\cdot g.\]
\end{lemma}

\begin{proof}
Let $g,h \in H$ and $a,b \in H^*$.  Then 
\begin{align*} 
\langle h\cdot (a\leftharpoonup g),b\rangle &= \langle a\leftharpoonup g, h\leftharpoonup b\rangle \\
&= \langle g,a\cdot (h\leftharpoonup b)\rangle\\
&=\left \langle g, \sum(h_1\leftharpoonup (S(h_2)\leftharpoonup a)\cdot b\right\rangle \text{ by Lemma~\ref{lem:product}}\\
&=\sum \langle g,h_1\leftharpoonup (S(h_2)\leftharpoonup a)\cdot b\rangle\\
&=\sum\langle h_1\cdot g,(S(h_2)\leftharpoonup a)\cdot b \rangle\\
&=\sum \langle (S(h_2)\leftharpoonup a)\leftharpoonup h_1\cdot g, b\rangle. \qedhere
\end{align*}
\end{proof}

We can use the right action to obtain an algebraic Littlewood-Richardson formula analogous to~\cite[\svw{Theorem 2.1}]{LLS2011} for those bases whose skew elements appear as the right tensor factor in the coproduct.  

Let $\{L_\alpha\} \subset H$ and $\{R_\beta\} \subset H^*$ be dual bases with indexing set $\mathcal{P}$.  Then 
\begin{align}
L_\alpha \cdot L_\beta = \sum_{\gamma}b^\gamma_{\alpha,\beta}L_\gamma & \qquad \Delta(L_\gamma) = \sum_{\alpha,\beta} c^\gamma_{\alpha,\beta} L_\alpha \otimes L_\beta \\
R_\alpha\cdot R_\beta = \sum_{\gamma}c^\gamma_{\alpha,\beta}R_\gamma & \qquad \Delta(R_\gamma) = \sum_{\alpha,\beta}b^\gamma_{\alpha, \beta} R_\alpha \otimes R_\beta
\end{align}
where $b^\gamma_{\alpha,\beta}$ and $c^\gamma_{\alpha,\beta}$ are structure constants.  We can also write 
\begin{equation}
\Delta(L_\gamma) = \sum_{\delta} L_\delta \otimes L_{\gamma/\delta} \qquad \Delta(R_\gamma) = \sum_{\delta} R_\delta \otimes R_{\gamma/\delta}.
\end{equation}

Note that $L_\alpha \leftharpoonup R_\beta = \svw{R_{\beta/\alpha}}$ and $R_\beta \leftharpoonup L_\alpha = L_{\alpha/\beta}$.  Further, 
\begin{equation}
\Delta(L_{\alpha/\beta}) = \sum_{\pi,\rho}c^\alpha_{\pi,\rho,\beta} L_\pi \otimes L_\rho \qquad \Delta(R_{\alpha/\beta}) = \sum_{\pi,\rho}b^\alpha_{\pi,\rho,\beta}R_\pi \otimes R_\rho.
\end{equation}  The antipode acts on $L_\rho$ by $S(L_\rho) = (-1)^{\theta(\rho)}L_{\rho^*}$ where $\theta:\mathcal{P}\rightarrow \mathbb{N}$ and $*:\mathcal{P}\rightarrow \mathcal{P}$.

\begin{theorem}\label{thm:genskewLR} For $\alpha, \beta,\gamma,\delta \in \mathcal{P}$, 
\[L_{\alpha/\beta}\cdot L_{\gamma/\delta} = \sum_{\pi,\rho,\nu,\mu}(-1)^{\theta(\rho)}c^\alpha_{\pi,\rho,\beta}b^\nu_{\pi,\gamma}b^\delta_{\mu,\rho^*} L_{\nu/\mu}.\]
\end{theorem}

\begin{proof} We use Lemma~\ref{lem:raction} and the preceding facts about the product, coproduct, and antipode maps on $H$ and $H^*$ to obtain
\begin{align*}
L_{\alpha/\beta}\cdot L_{\gamma/\delta} & = L_{\alpha/\beta}\cdot (R_\delta \leftharpoonup L_\gamma)\\
&=\sum_{\pi,\rho} c^\alpha_{\pi,\rho,\beta}(S(L_\rho)\leftharpoonup R_\delta) \leftharpoonup (L_\pi\cdot L_\gamma)\\
&= \sum_{\pi,\rho}(-1)^{\theta(\rho)}c^\alpha_{\pi,\rho,\beta}(L_{\rho^*} \leftharpoonup R_\delta) \leftharpoonup(L_\pi\cdot L_\gamma)\\
&=\sum_{\pi,\rho}(-1)^{\theta(\rho)}c^\alpha_{\pi,\rho,\beta}\left(R_{\delta/\rho^*} \leftharpoonup \left(\sum_{\nu}b^\nu_{\pi,\gamma}L_\nu\right)\right)\\
&=\sum_{\pi,\rho,\nu}(-1)^{\theta(\rho)}c^\alpha_{\pi,\rho,\beta}b^\nu_{\pi,\gamma} (R_{\delta/\rho^*}\leftharpoonup L_\nu)\\
&=\sum_{\pi,\rho,\nu,\mu}(-1)^{\theta(\rho)}c^\alpha_{\pi,\rho,\beta}b^\nu_{\pi,\gamma} b^\delta_{\mu,\rho^*}(R_\mu\leftharpoonup L_\nu)\\
&=\sum_{\pi,\rho,\nu,\mu}(-1)^{\theta(\rho)}c^\alpha_{\pi,\rho,\beta}b^\nu_{\pi,\gamma} b^\delta_{\mu,\rho^*} L_{\nu/\mu}. \qedhere
\end{align*}
\end{proof}

\section{The dual Hopf algebras $\Qsym$ and $\Nsym$}\label{sec:QSymNSym}
We now focus our attention on the dual Hopf algebra pair of noncommutative symmetric functions and quasisymmetric functions, and introduce our main objects of study the (row-strict) dual immaculate functions.

A \emph{composition} $\alpha = (\alpha _1, \ldots, \alpha _k)$ of $n$, denoted by $\alpha \vDash n$ is a list of positive integers such that $\sum _ {i=1} ^{k} \alpha _i = n$. We call $n$ the \emph{size} of $\alpha$ and sometimes denote it by $|\alpha |$, and call $k$ the \emph{length} of $\alpha$ and sometimes denote it by $\ell (\alpha)$. If $\alpha _{j_1}= \cdots = \alpha _{j_m} = i$ we sometimes abbreviate this to $i^m$, and denote the \emph{empty composition} of 0 by $\emptyset$. 
There exists a natural correspondence between compositions $\alpha \vDash n$ and subsets $S\subseteq \{ 1, \ldots , n-1\} = [n-1]$. More precisely, $\alpha = (\alpha _1, \ldots , \alpha _k)$ corresponds to $\set (\alpha) = \{ \alpha _1, \alpha _1 + \alpha _2, \ldots , \alpha _1 + \cdots +\alpha _{k-1}\}$, and conversely $S= \{ s_1, \ldots , s_{k-1}\}$ corresponds to $\comp (S) = ( s_1, s_2 - s_1, \ldots , n- s_{k-1})$. We also denote by $S^c$ the set complement of $S$ in $[n-1]$. 

Given a composition $\alpha$, its \emph{diagram}, also denoted by $\alpha$, is the array of left-justified boxes with $\alpha _i$ boxes in row $i$ from the \emph{bottom}. Given two compositions $\alpha, \beta$ we say that $\beta \subseteq \alpha$ if $\beta _j \leq \alpha _j$ for all $1\leq j \leq \ell (\beta) \leq \ell (\alpha)$, and given $\alpha, \beta$ such that $\beta \subseteq \alpha$, the \emph {skew diagram}  $\alpha / \beta$ is the array of boxes in $\alpha$ but not $\beta$ when $\beta$ is placed in the bottom-left corner of $\alpha$. If, furthermore, $\beta \subseteq \alpha$ and $\alpha _j - \beta _j \in \{ 0,1\}$ for all $1\leq j \leq \ell (\beta)\leq \ell (\alpha) $ then we call $\alpha / \beta$ a \emph{vertical strip}.  

\begin{example}\label{ex:diags}
If $\alpha = (3,4,1)$, then $|\alpha|=8, \ell (\alpha) = 3$, and $\set (\alpha) = \{3,7\}$. Its diagram is 
$$\alpha = \tableau{ \ \\ \ & \ & \ & \ \\ \ & \ & \ }$$and if $\beta = (2,4)$, then $$\alpha / \beta = \tableau{ \ \\   &   &   &   \\  &  & \ }$$is a vertical strip.
\end{example}

\begin{definition}\label{def:SIT}
Given a composition $\alpha$, a \emph{standard immaculate tableau} $T$ of \emph{shape} $\alpha$ is a bijective filling of its diagram with $1, \ldots , |\alpha|$ such that
\begin{enumerate}
\item The entries in the leftmost column increase from bottom to top;
\item The entries in each row increase from left to right.
\end{enumerate}We obtain a \emph{standard skew immaculate tableau} of shape $\alpha / \beta$ by extending the definition to skew diagrams $\alpha / \beta$ in the natural way. 
\end{definition}

Given a standard (skew) immaculate tableau, $T$, its \emph{descent set} is
$$\Des (T) = \{ i : i+1 \mbox{ appears strictly above $i$ in $T$} \}.$$

\begin{example}\label{ex: skewSIT}
A standard skew immaculate tableau of shape $(3,4,1)/(1)$ is
$$T = \tableau{7\\ 2&3&4&6\\ & 1& 5}$$with $\Des (T) = \{ 1,5,6 \}$.
\end{example}

We are now ready to define our Hopf algebras and functions of central interest.

Given a composition $\alpha = ( \alpha _1, \ldots , \alpha _k) \vDash n$ and commuting variables $\{ x_1, x_2, \ldots \}$ we define the \emph{monomial quasisymmetric function} $M_\alpha$ to be
$$M_\alpha = \sum _{i_1 < \cdots < i_k} x_{i_1} ^{\alpha _1} \cdots x_{i_k} ^{\alpha _k}$$ the \emph{fundamental quasisymmetric function} $F_\alpha$ to be
$$F_\alpha = \sum _{i_1 \leq \cdots \leq i_n \atop i_j = i_{j+1} \Rightarrow j \not\in \set(\alpha)} x_{i_1} \cdots x_{i_n} $$the \emph{dual immaculate function} $\dI_\alpha$ to be
$$\dI _\alpha = \sum _{T} F_{\comp(\Des(T))}$$and the 
\emph{row-strict dual immaculate function} $\rdI _\alpha$ to be
$$\rdI _\alpha = \sum _{T} F_{\comp(\Des(T)^c)}$$where the latter two sums are over all standard immaculate tableaux $T$ of shape $\alpha$. These extend naturally to give \emph{skew} dual immaculate and row-strict dual immaculate functions $\dI _{\alpha / \beta}$ \cite{BBSSZ2014} $\rdI _{\alpha / \beta}$ \cite{NSvWVW2022}, where $\alpha / \beta$ is a skew diagram.

The set of all monomial or fundamental quasisymmetric functions forms a basis for the \emph{Hopf algebra of quasisymmetric functions}  $\QSym$, as do the set of all (row-strict) dual immaculate functions. There exists an involutory automorphism $\psi$ defined on fundamental quasisymmetric functions by
$$\psi (F_\alpha) = F_{\comp (\set (\alpha ^c))}$$such that \cite{NSvWVW2022}
$$\psi (\dI _\alpha) = \rdI _\alpha$$for a composition $\alpha$. This extends naturally to skew diagrams $\alpha / \beta$ to give
$$\psi (\dI _{\alpha/\beta}) = \rdI _{\alpha/\beta}.$$

Dual to the Hopf algebra of quasisymmetric functions is the \emph{Hopf algebra of noncommutative symmetric funtions} $\Nsym$. Given a composition $\alpha = (\alpha _1, \ldots, \alpha _k) \vDash n$ and \emph{noncommuting variables} $\{ y_1, y_2, \ldots \}$ we define the  \emph{$n$th elementary noncommutative symmetric function} $\nce _n$ to be
$$\nce _n = \sum _{i_1<\cdots < i_n} y_{i_1}\cdots y_{i_n}$$and the \emph{elementary noncommutative symmetric function} $\nce _\alpha$ to be
$$\nce _\alpha = \nce _{\alpha _1} \cdots \nce _{\alpha _k}.$$ Meanwhile, we define the \emph{$n$th complete homogeneous noncommutative symmetric function} $\nch _n$ to be
$$\nch _n = \sum _{i_1\leq \cdots \leq i_n}y_{i_1}\cdots y_{i_n}$$and the \emph{complete homogeneous noncommutative symmetric function} $\nch _\alpha$ to be
$$\nch_\alpha = \nch _{\alpha _1} \cdots \nch _{\alpha _k}.$$ The set of all elementary or complete homogeneous noncommutative symmetric functions forms a basis for $\Nsym$. The duality between $\Qsym$ and $\Nsym$ is given by
$$\langle M_\alpha , \nch _\alpha \rangle = \delta _{\alpha\beta}$$where $\delta _{\alpha\beta}=1$ if $\alpha = \beta$ and 0 otherwise. This induces the bases dual to the (row-strict) dual immaculate functions via
$$\langle \dI _\alpha , \nci _\alpha \rangle = \delta _{\alpha\beta} \qquad \langle \rdI_\alpha , \ncri _\alpha \rangle = \delta _{\alpha\beta}$$and  implicitly defines the bases of \emph{immaculate} and \emph{row-strict immaculate functions}. While concrete combinatorial definitions of these functions have been established \cite{BBSSZ2014, NSvWVW2022}, we will not need them here. However, what we will need is the involutory automorphism in $\Nsym$ corresponding to $\psi$  in $\Qsym$, defined by $\psi (\nce _\alpha) = \nch _\alpha$ that gives \cite{NSvWVW2022} $\psi (\nci _\alpha ) = \ncri _\alpha$.

\section{The Pieri rules for skew dual immaculate functions}\label{sec:Pieri}
A left Pieri rule for immaculate functions was conjectured in~\cite[Conjecture 3.7]{BBSSZ2014} and proved in~\cite{BSOZ2016}. Given a composition $\alpha = (\alpha_1,\ldots,\alpha_k)$ we say that $\tail(\alpha) = (\alpha_2,\ldots,\alpha_k)$. If $\beta\in \mathbb{Z}^k$, then $\negc(\alpha-\beta) = |\{i:\alpha_i-\beta_i<0\}|$. Let $\sgn(\beta)=(-1)^{\negc(\beta)}$ with $\negc(\beta)=|\{i:\beta_i<0\}|$.

Following~\cite{BSOZ2016}, we define $Z_{s,\alpha}$ to be a set of all $\beta\in \mathbb{Z}^k$ such that \begin{enumerate}
    \item $\beta_1+\cdots+\beta_k=s$ and $\beta_1+\cdots+\beta_i\leq s$ for all $i<k$; \\
    \item $\alpha_i-\beta_i\geq 0$ for all $1\leq i\leq k$ and $|i:\alpha_i-\beta_i=0|\leq 1$; \\
    \item For all $1\leq i\leq k$,
    \begin{itemize}
    \item if $\alpha_i>s-(\beta_1+\cdots+\beta_{i-1}),$ then $ 0\leq \beta_i\leq s-(\beta_1+\cdots+\beta_{i-1}),$
    \item if $\alpha_i<s-(\beta_1+\cdots+\beta_{i-1})$, then $\beta_i<0$, and
    \item if $\alpha_i=s-(\beta_1+\cdots+\beta_{i-1}),$ then either $\beta_i<0$ or $\beta_i=\alpha_i$ and $\beta_{i+1}=\cdots=\beta_k=0$.
    \end{itemize}
    \end{enumerate}
    
Now we are ready to define the coefficients of the immaculate basis appearing in the left Pieri rule.

\begin{definition}[\cite{BSOZ2016}]\label{piericoeff}
For a positive integer $s$ and compositions $\alpha, \gamma$ with $|\alpha|-|\gamma|=s$, let $1\leq j\leq k$ be the smallest integer such that $\alpha_i=\gamma_{i-1}$ for all $j<i\leq k$ where $j=k$ when $\alpha_k\neq \gamma_{k-1}$. Let $j\leq r\leq k$ be the largest integer such that $\alpha_j<\alpha_{j+1}<\cdots<\alpha_r$. Let $\alpha^{(i)} = (\alpha_1,\ldots,\alpha_i)$ Then define \[c^{\gamma}_{s,\alpha}=
\left\{\begin{array}{ll}\sgn(\alpha-\gamma),&\text{if } \ell(\gamma)=\ell(\alpha)  \text{ and }  \alpha-\gamma\in Z_{s,\alpha};\\
\sgn(\alpha^{(j-1)}-\gamma^{(j-1)}) & \text{if } \ell(\gamma)=\ell(\alpha)-1,\\ &r-j  \text{ is even, and } \\ &(\alpha^{(j-1)}-\gamma^{(j-1)},\alpha_j,0,\ldots,0)\in Z_{s,\alpha};
 \\
0&\text{otherwise.}\end{array}\right.\]
\end{definition}

\begin{theorem}[\cite{BBSSZ2014,BSOZ2016}]\label{thm:immleftpieri}
Let $m>0$ and $\alpha$ be a composition.  Then 
\[\nch_m\nci_\alpha = \sum_{\substack{\beta\vDash |\alpha|+m\\\beta_1\geq m\\0\leq \ell(\beta)-\ell(\alpha)\leq 1}} c^{\tail(\beta)}_{\beta_1-m, \alpha} \nci_\beta.\]
\end{theorem}
Applying $\psi$ to both sides of the left Pieri rule in Theorem~\ref{thm:immleftpieri} immediately yields a left Pieri rule for row-strict immaculate functions.

\begin{corollary}\label{cor:rsimmleftpieri}
Let $m>0$ and $\alpha$ be a composition.  Then 
\[\nce_m \ncri_\alpha =  \sum_{\substack{\beta\vDash |\alpha|+m\\\beta_1\geq m\\0\leq \ell(\beta)-\ell(\alpha)\leq 1}}c^{\tail(\beta)}_{\beta_1-m, \alpha} \ncri_\beta.\]
\end{corollary}


Lemma 3.1 of~\cite{BSOZ2016} shows that for $s\geq 0$, $r>0$ and compositions $\alpha,\beta$ with $|\alpha|=|\beta|+s$, 
\[\langle \nci_\alpha,\svw{F_{(s)}}\dI_\beta\rangle=\langle \nch_r\nci_\alpha,\dI_{(s+r,\beta)}\rangle.\]  This leads to the following Pieri rule for dual immaculate functions. 

\begin{theorem}[\cite{BSOZ2016}]\label{thm:dimmleftpieri}
Let $s>0$ and $\alpha$ be a composition.  Then 
\[\svw{F_{(s)}}\dI_\alpha = \sum_{\substack{\beta\vDash |\alpha|+s\\0\leq \ell(\beta)-\ell(\alpha)\leq 1}} c^{\alpha}_{s,\beta} \dI_\beta.\]
\end{theorem}

Again, applying $\psi$ to both sides gives a Pieri rule for row-strict dual immaculate functions.

\begin{corollary}\label{cor:rsdimmleftpieri}
Let $s>0$ and $\alpha$ be a composition.  Then 
\[F_{(1^s)} \rdI_\alpha = \sum_{\substack{\beta\vDash |\alpha|+s\\0\leq \ell(\beta)-\ell(\alpha)\leq 1}} c^{\alpha}_{s,\beta} \rdI_\beta.\]
\end{corollary}

We use these results together with Hopf algebra computations to construct a Pieri rule for skew dual immaculate functions.  Using the  map $\psi$, this also gives a Pieri rule for row-strict skew dual immaculate functions. But first we have a small, yet crucial, lemma.

\begin{lemma}\label{lem:immract}
 Let $\alpha$ and $\gamma$ be compositions.  Then $\nci_\gamma \leftharpoonup \dI_\alpha = \dI_{\alpha/\gamma}$.  
 \end{lemma}
 \begin{proof}
Recall that if $H= \QSym$ and $H^*=\NSym$ are our pair of dual Hopf algebras, then we know $\Delta \dI_\alpha =\sum_\beta \dI_\beta \otimes \dI_{\alpha/\beta}$ and we have
 \[\nci_\gamma \leftharpoonup \dI_\alpha = \sum_\beta \langle \nci_\gamma , \dI_\beta \rangle \dI_{\alpha/\beta} = \dI_{\alpha/\gamma}\]since $\langle\nci_\gamma , \dI_\beta \rangle = \delta_{\gamma\beta}$, where $\delta_{\gamma\beta} = 1$ if $\gamma=\beta$ and 0 otherwise.
\end{proof}

We can now give our Pieri rule for (row-strict) skew dual immaculate functions.

\begin{theorem}\label{thm:skewpieri} Let $\gamma\subseteq\alpha$.  Then
\[\dI_{(s)}\dI_{\alpha/\gamma}=\sum_{\beta/\tau} (-1)^{|\gamma|-|\tau|}\cdot c^{\alpha}_{|\beta|-|\alpha|,\beta}\,\dI_{\beta/\tau}\]and hence by applying $\psi$ to both sides
\svw{\[\rdI_{(s)}\rdI_{\alpha/\gamma}=\sum_{\beta/\tau} (-1)^{|\gamma|-|\tau|}\cdot c^{\alpha}_{|\beta|-|\alpha|,\beta}\,\rdI_{\beta/\tau}\]}where $|\beta/\tau|=|\alpha/\gamma|+s$, $\gamma/\tau$ is a vertical strip of length at most $s$, $\ell(\beta)-\ell(\alpha) \in \{0,1\}$ and $c^{\alpha}_{|\beta|-|\alpha|,\beta}$ is the coefficient of Definition \ref{piericoeff}. These decompositions are multiplicity-free up to sign.
\end{theorem}

\begin{proof}
Note that $\dI_{(1^s)}=F_{(1^s)}$  and $\dI_{(s)}=F_{(s)}$. Recall that 
\begin{equation}\label{eq:fundcoprod}
\Delta F_\alpha = \sum_{\substack{(\beta,\gamma) \text{ with }\\\beta\cdot \gamma = \alpha \text{ or}\\\beta\odot \gamma=\alpha}} F_\beta\otimes F_\gamma
\end{equation}
where for $\beta = (\beta_1,\ldots, \beta_k)$ and $\gamma = (\gamma_1,\ldots, \gamma_l)$, $\beta\cdot\gamma = (\beta_1,\ldots, \beta_k,\gamma_1,\ldots, \gamma_l)$ is the \emph{concatenation} of $\beta$ and $\gamma$,  and $\beta\odot\gamma = (\beta_1, \ldots, \beta_{k-1},\beta_k+\gamma_1,\gamma_2,\ldots, \gamma_l)$ is the \emph{near-concatenation} of $\beta$ and $\gamma$. 

Then we have
\[\Delta(F_{(s)})=\sum_{i=0}^sF_{(i)}\otimes \svw{F_{(s-i)}.}\]
Thus,
\begin{align*}
\dI_{(s)}\dI_{\alpha/\gamma}   
&=\dI_{(s)}(\I_\gamma\leftharpoonup\dI_\alpha)\,\,\,\,\,\,\,\,\,\,\,\,\,\,\,\,\, \text{by Lemma~\ref{lem:immract}}\\
&=F_{(s)}(\I_\gamma\leftharpoonup\dI_\alpha) \\
&=\sum_{i=0}^s(S(\svw{F_{(s-i)}})\leftharpoonup \I_\gamma)\leftharpoonup(\svw{F_{(i)}}\dI_\alpha)\,\,\,\,\,\,\, \text{by Lemma~\ref{lem:raction}.} \\
\end{align*}

We first compute  $\sw{S(F_{(s-i)})}\leftharpoonup \I_\gamma$. Since it is well known that $S(F_\alpha)=(-1)^{|\alpha|}F_{\comp(\set(\alpha)^c)}$  we have $\sw{S(F_{(s-i)})=(-1)^{s-i}F_{(1^{s-i})}}$. Furthermore, we can write the coproduct as \[\Delta(\I_\gamma)=\sum_{\delta,\tau}b^\gamma_{\delta,\tau}\I_\delta\otimes\I_\tau.\] Thus,
\sw{\begin{align*}
S(F_{(s-i)})\leftharpoonup \I_\gamma
&=(-1)^{s-i}F_{(1^{s-i})}\leftharpoonup \I_\gamma \\
&=\sum_{\delta,\tau}(-1)^{s-i}b^\gamma_{\delta,\tau}\langle F_{(1^{s-i})}, \I_\delta\rangle\I_\tau \\
&=\sum_{\delta,\tau}(-1)^{s-i}b^\gamma_{\delta,\tau}\langle\I^*_{(1^{s-i})}, \I_\delta\rangle\I_\tau \\
&=\sum_{\tau}(-1)^{s-i}b^\gamma_{(1^{s-i}), \tau}\I_\tau. 
\end{align*}}
By the definition of product and coproduct on $\Nsym$, we have \[b^\gamma_{\delta,\tau}=\langle \Delta\I_\gamma,\dI_\delta\otimes\dI_\tau\rangle=\langle\I_\gamma,\dI_\delta\cdot\dI_\tau\rangle\svw{.}\]
To compute this for $\sw{\delta = (1^{s-i})}$ we use Proposition 3.34 from~\cite{BBSSZ2014} which states that $F^\perp_{(1^r)}\nci_\alpha = \sum_\beta \nci_\beta$ where $\beta \in \mathbb{Z}^{\ell(\alpha)}$, \svw{$\alpha_k - \beta_k \in \{0,1\}$} for all $k$ and $|\beta| = |\alpha|-r$.  The operator $F^\perp$ is used throughout~\cite{BBSSZ2014}, and has the property that $\langle F^\perp \nci_\alpha,\dI_\beta \rangle = \langle \nci_\alpha,F \dI_\beta\rangle$.  

Thus, 
\sw{\begin{align*}
b^\gamma_{(1^{s-i}), \tau}
&=\langle \nci_\gamma, \dI_{(1^{s-i})}\dI_\tau\rangle\\
&=\langle\nci_\gamma,F_{(1^{s-i})}\dI_\tau\rangle\\
&=\langle F^{\perp}_{(1^{s-i})}\nci_\gamma, \dI_\tau\rangle\\
&=\left\langle \sum_{\beta}\nci_\beta, \dI_\tau\right\rangle\\
&=\delta _{\beta \tau}
\end{align*}}where the sum is over all $\beta$ such that $\beta \in \mathbb{Z}^{\ell(\gamma)}$, $\gamma_k - \beta_k \in \{0,1\}$ for all $k$, and $|\beta|=|\gamma|-\sw{(s-i)}$.  

Then using the above calculations, Theorem~\ref{thm:dimmleftpieri} and Lemma~\ref{lem:immract}, we have
\begin{align*}
\dI_{(s)}\dI_{\alpha/\gamma}   
&=\dI_{(s)}(\I_\gamma\leftharpoonup\dI_\alpha)\\
&=\sum_{i=0}^s\left((S(F_{(s-i)})\leftharpoonup \I_\gamma)\leftharpoonup(F_{(i)}\dI_\alpha)\right)\\
&=\sum_{i=0}^s \left((-1)^{(s-i)}\sum_{\substack{\tau \in \mathbb{Z}^{\ell(\gamma)}\\\gamma_k-\tau_k \in \{0,1\}\\|\tau|=|\gamma|-(s-i)}} \nci_\tau\right)\leftharpoonup \left(\sum_{\substack{\beta\vDash |\alpha|+i\\0\leq\ell(\beta) -\ell(\alpha)\leq 1}} c^{\alpha}_{i,\beta} \dI_{\beta}\right)\\
&=\sum_{i=0}^s \sum_{\substack{\tau,\beta\\
\tau \in \mathbb{Z}^{\ell(\gamma)}\\\gamma_k-\tau_k\in\{0,1\}\\|\tau|=|\gamma|-(s-i)\\\beta\vDash|\alpha|+i\\ \ell(\beta) - \ell(\alpha) \in \{0,1\}}} (-1)^{(s-i)}\cdot c^{\alpha}_{i,\beta}\, \dI_{\beta/\tau}\\
&=\sum_{\beta/\tau} (-1)^{|\gamma|-|\tau|}\cdot c^{\alpha}_{|\beta|-|\alpha|,\beta}\,\dI_{\beta/\tau}
\end{align*}
where $|\beta/\tau|=|\alpha/\gamma|+s$, $\gamma/\tau$ is a vertical strip of length at most $s$, and $\ell(\beta)-\ell(\alpha) \in \{0,1\}$.
\end{proof}

\begin{example}
Let us compute $\dI_{(2)}\cdot\dI_{(1,2,1)/(1,1)}$.

First, we need to compute all compositions $\beta\vDash 4+i$ for $i\in \{0,1,2\}$ and $\ell(\beta)=3$ or $4$. We list all possible choices for $\beta$ as the set 
\begin{align*}
A=\{&(1,1,1,1),(1,1,2),(1,2,1),(2,1,1),(1,1,1,2),(1,1,2,1),(1,2,1,1),\\
&(2,1,1,1),(1,1,3),(1,2,2),(1,3,1),(2,1,2),(2,2,1),(3,1,1),(1,1,1,3),\\
&(1,1,2,2),(1,1,3,1),(1,2,1,2),(1,2,2,1),(1,3,1,1),(2,1,1,2),\\
&(2,1,2,1),(2,2,1,1),(3,1,1,1),(1,1,4),(1,2,3),(1,3,2),(1,4,1),\\
&(2,1,3),(2,2,2),(2,3,1),(3,1,2),(3,2,1),(4,1,1)\}.
\end{align*}
 
Next we need to find $\tau$ by removing a vertical strip of length at most \svw{$s=2$} from $\gamma=(1,1)$. We list all options for $\tau$ as the set $B=\{\emptyset,(1),(1,1)\}$.

By Theorem~\ref{thm:skewpieri}, now we expand $\dI_{(2)}\cdot\dI_{(1,2,1)/(1,1)}$ by finding all valid pairs $(\beta,\tau)$ such that $|\beta/\tau|=4$. Thus \begin{equation*}
    \begin{aligned}
    \dI_{(2)}\cdot\dI_{(1,2,1)/(1,1)}
    &=c^{(1,2,1)}_{0,(1,1,1,1)}\dI_{(1,1,1,1)}+c^{(1,2,1)}_{0,(1,1,2)}\dI_{(1,1,2)}\\&+c^{(1,2,1)}_{0,(1,2,1)}\dI_{(1,2,1)}
    +c^{(1,2,1)}_{0,(2,1,1)}\dI_{(2,1,1)}\\&-c^{(1,2,1)}_{1,(1,1,1,2)}\dI_{(1,1,1,2)/(1)}-c^{(1,2,1)}_{1,(1,1,2,1)}\dI_{(1,1,2,1)/(1)}\\
    &-c^{(1,2,1)}_{1,(1,2,1,1)}\dI_{(1,2,1,1)/(1)}-c^{(1,2,1)}_{1,(2,1,1,1)}\dI_{(2,1,1,1)/(1)}\\&-c^{(1,2,1)}_{1,(1,1,3)}\dI_{(1,1,3)/(1)}
    -c^{(1,2,1)}_{1,(1,2,2)}\dI_{(1,2,2)/(1)}\\&-c^{(1,2,1)}_{1,(1,3,1)}\dI_{(1,3,1)/(1)}-c^{(1,2,1)}_{1,(2,1,2)}\dI_{(2,1,2)/(1)}\\
    &-c^{(1,2,1)}_{1,(2,2,1)}\dI_{(2,2,1)/(1)}-c^{(1,2,1)}_{1,(3,1,1)}\dI_{(3,1,1)/(1)}\\&+c^{(1,2,1)}_{2,(1,1,1,3)}\dI_{(1,1,1,3)/(1,1)}+c^{(1,2,1)}_{2,(1,1,2,2)}\dI_{(1,1,2,2)/(1,1)}\\&+c^{(1,2,1)}_{2,(1,1,3,1)}\dI_{(1,1,3,1)/(1,1)}+c^{(1,2,1)}_{2,(1,2,1,2)}\dI_{(1,2,1,2)/(1,1)}\\&+c^{(1,2,1)}_{2,(1,2,2,1)}\dI_{(1,2,2,1)/(1,1)}+c^{(1,2,1)}_{2,(1,3,1,1)}\dI_{(1,3,1,1)/(1,1)}\\&+c^{(1,2,1)}_{2,(2,1,1,2)}\dI_{(2,1,1,2)/(1,1)}+c^{(1,2,1)}_{2,(2,1,2,1)}\dI_{(2,1,2,1)/(1,1)}\\&+c^{(1,2,1)}_{2,(2,2,1,1)}\dI_{(2,2,1,1)/(1,1)}+c^{(1,2,1)}_{2,(3,1,1,1)}\dI_{(3,1,1,1)/(1,1)}\\&+c^{(1,2,1)}_{2,(1,1,4)}\dI_{(1,1,4)/(1,1)}+c^{(1,2,1)}_{2,(1,2,3)}\dI_{(1,2,3)/(1,1)}\\&+c^{(1,2,1)}_{2,(1,3,2)}\dI_{(1,3,2)/(1,1)}+c^{(1,2,1)}_{2,(1,4,1)}\dI_{(1,4,1)/(1,1)}\\
    &+c^{(1,2,1)}_{2,(2,1,3)}\dI_{(2,1,3)/(1,1)}+c^{(1,2,1)}_{2,(2,2,2)}\dI_{(2,2,2)/(1,1)}\\&+c^{(1,2,1)}_{2,(2,3,1)}\dI_{(2,3,1)/(1,1)}+c^{(1,2,1)}_{2,(3,1,2)}\dI_{(3,1,2)/(1,1)}\\&+c^{(1,2,1)}_{2,(3,2,1)}\dI_{(3,2,1)/(1,1)}+c^{(1,2,1)}_{2,(4,1,1)}\dI_{(4,1,1)/(1,1)}.
    \end{aligned}
\end{equation*}
We can compute all the coefficients $c^{\alpha}_{|\beta|-|\alpha|,\beta}$ and most of them turn out to be zero. Hence we have the following expansion after simplification.
\begin{equation*}
    \begin{aligned}
    \dI_{(2)}\cdot\dI_{(1,2,1)/(1,1)}
    &=\dI_{(1,2,1)}-\dI_{(1,1,2,1)/(1)}-\dI_{(2,2,1)/(1)}+\dI_{(2,1,2,1)/(1,1)}\\
    &+\svw{\dI_{(3,2,1)/(1,1)}}
    \end{aligned}
\end{equation*}
\end{example}

\bibliographystyle{plain}
\bibliography{refsra}

\end{document}